\documentclass[a4paper,12pt]{article}
\usepackage[top=2.5cm,bottom=2.5cm,left=2.5cm,right=2.5cm]{geometry}
\usepackage{cite, amsmath, amssymb, amsthm}
\usepackage{tikz}
\usepackage{mathrsfs}
\usetikzlibrary{arrows}
\usepackage[margin=1cm,%
font=small,%
format=hang,%
labelsep=period,%
labelfont=bf]{caption}
\newtheorem{theorem}{Theorem}[section]
\newtheorem{defin}[theorem]{Definition}
\newtheorem{lemma}[theorem]{Lemma}

\newtheorem{conj}{Conjecture}

\newtheorem{nt}{Note}

\newcommand{\singlespacing}{\let\CS=\@currsize\renewcommand{\baselinestretch}{1}\tiny\CS}
\newcommand{\oneandahalfspacing}{\let\CS=\@currsize\renewcommand{\baselinestretch}{1.25}\tiny\CS}
\newcommand{\doublespacing}{\let\CS=\@currsize\renewcommand{\baselinestretch}{1.35}\tiny\CS}

\newtheorem{proposition}[theorem]{Proposition}

\newtheorem{rule-def}[theorem]{Rule}

\begin{document}

\baselineskip 16pt

\vspace*{3cm}

\begin{center}
 {\Large \bf On Sombor Index of Graphs}\\

  \vspace{8mm}

{\bf Batmend Horoldagva$^{a,}\footnote{Corresponding
		author}$,  Chunlei Xu$^{a, b}$ }

 \vspace{6mm}

 \baselineskip=0.20in

 $^a${\it Department of Mathematics, Mongolian National University of Education,\\
 Baga toiruu-14, Ulaanbaatar 48, Mongolia\/}\\
  {\tt horoldagva@msue.edu.mn}\\[2mm]

 $^b${\it School of Mathematics and Physics, Inner Mongolia University for Nationalities, \\
 Tongliao, People's Republic of China\/} \\
 {\tt xuchunlei1981@sina.cn}\\[2mm]

 \vspace{3mm}
 \vspace{3mm}
 \end{center}
 
 \baselineskip=0.23in
\begin{abstract}

\vspace{3mm}
Recently, Gutman  defined a new vertex-degree-based graph invariant, named the Sombor index $SO$ of a graph $G$, and   is defined by
$$SO(G)=\sum_{uv\in E(G)}\sqrt{d_G(u)^2+d_G(v)^2},$$
where $d_G(v)$ is the degree of the vertex $v$ of  $G$.  
In this paper, we obtain the sharp lower and upper bounds on $SO(G)$ of a connected graph, and characterize graphs for which these bounds are attained.

\bigskip

\end{abstract}

\section{Introduction}

Let $G$  be a connected graph with vertex set $V(G)$ and edge set $E(G)$. The order of  $G$ is denoted by $n$.  The degree
of the vertex $v$ is denoted by  $d_G(v)$.  
For $v\in V(G)$, $N_G(v)$ denotes the set of all neighbors of $v$. An edge $uv$  of a graph $G$ is called a  cut edge  if the graph $G-uv$ is disconnected.   For $uv\in E(G)$, denote by $G-uv$ the subgraph of $G$ obtained from $G$ by deleting the edge $uv$.   For two nonadjacent vertices $u$ and $v$ of $G$, denote by $G+uv$  the graph obtained from $G$ by adding the edge $uv$. The girth of a graph $G$ is the length of the shortest cycle which is contained in $G$. The maximum degree of $G$ is denoted by $\Delta$. The complete graph and the cycle of order $n$ are denoted by $K_n$ and $C_n$, respectively. The clique number of a graph $G$ is the maximal order of a complete
subgraph of $G$.

\vspace{3mm}
Gutman \cite{G}  defined a new vertex-degree-based graph invariant, named "Sombor index" of a graph $G$, denoted by $SO(G)$ and is defined by
$$SO(G)=\sum_{uv\in E(G)}\sqrt{d_G(u)^2+d_G(v)^2}.$$
Mathematical  properties and applications of $SO$ index were established in \cite{G}.

\vspace{3mm}
In this paper, we obtain the sharp lower bounds  on $SO(G)$ of a  graph of order $n$ with the maximum degree $\Delta$ and of a graph of order $n$ with girth $g$. 
Also, we give the sharp upper  bound  on $SO(G)$ of a unicyclic  graph of order $n$ with girth $g$.  Very recently, for the  graphs of  order $n$ with $k$  pendent vertices, the graphs were characterized that have the extremal classical Zagreb indices \cite{ET}, multiplicative sum Zagreb index \cite{HXBD}, and reduced second Zagreb index \cite{HSBD}.  
Hence, furthermore, we obtain the sharp upper bound  on $SO(G)$ of a  graph of order $n$ with $k$ pendent vertices ($r$ cut edges).  
Moreover, the corresponding extremal graphs are characterized   for which all the above bounds are attained.

\vspace{3mm}
\section{Graphs with minimum Sombor  index}
In this section, we study the graphs with minimum Sombor index.  Let  $P=uu_1u_2\cdots u_k$ be a path of length $k$ in $G$ such that $d_G(u)\geq 3$, $d_G(u_k)=1$ and $d_G(u_i)=2$ for $i=1,2,\dots, k-1$. Then it is called
 a pendent path in $G$,
$u$ and $k$ are  called the origin and the length of $P$. 
Let  us consider a function $\theta(t)=\sqrt{t^2+4}-\sqrt{t^2+1}$ and one can easily see  that  $\theta(t)$ is decreasing on   $[0, +\infty)$.

 \begin{lemma}\label{lemma31} Let $P$ and $Q$ be two pendent paths with origins $u$ and $v$ in graph $G$, respectively.  Let $x$ be a neighbor vertex of $u$ who lies on  $P$  and $y$ be the pendent vertex on  $Q$.  Denote $G^\prime =G-ux+xy$.  
  Then $SO(G)>SO(G^\prime)$.
\end{lemma}
\begin{proof} Let $z$ be the neighbor vertex of $y$ in $G$. Suppose first that
 $u\neq v$. Then
\begin{eqnarray}
&&SO(G)-SO(G^\prime)\nonumber\\
&&=\sum_{w\in N_G(u)\setminus x}\sqrt{d_G(u)^2+d_G(w)^2}+
\sqrt{d_G(u)^2+d_G(x)^2}+\sqrt{1+d_G(z)^2}\nonumber\\
&&-\sum_{w\in N_G(u)\setminus x}\sqrt{(d_G(u)-1)^2+d_G(w)^2}-
\sqrt{2^2+d_G(x)^2}-\sqrt{2^2+d_G(z)^2}\nonumber\\
&&>\sqrt{d_G(u)^2+d_G(x)^2}+\sqrt{1+d_G(z)^2}-\sqrt{2^2+d_G(x)^2}-\sqrt{2^2+d_G(z)^2}.\label{e1}
\end{eqnarray}

Suppose now that $u=v$.  If  the length of $Q$ is equal to one, then $u=z$ and
\begin{eqnarray}
&&SO(G)-SO(G^\prime)\nonumber\\
&&=\sum_{w\in N_G(u)\setminus \{x, y\}}\sqrt{d_G(u)^2+d_G(w)^2}+\sqrt{d_G(u)^2+1}+\sqrt{d_G(u)^2+d_G(x)^2}\nonumber\\
&&-\sum_{w\in N_G(u)\setminus \{x, y\}}\sqrt{(d_G(u)-1)^2+d_G(w)^2}-\sqrt{(d_G(u)-1)^2+2^2}-\sqrt{2^2+d_G(x)^2}\nonumber\\
&&>\sqrt{d_G(u)^2+d_G(x)^2}-\sqrt{2^2+d_G(x)^2}\geq 0\nonumber
\end{eqnarray}
since $d_G(u)\geq 2$. If the length of $Q$ is greater than one and let $x^\prime$ be the neighbor of $u$ on path $Q$. Then
\begin{eqnarray}
&&SO(G)-SO(G^\prime)\nonumber\\
&&=\sum_{w\in N_G(u)\setminus \{x,x^\prime\}}\sqrt{d_G(u)^2+d_G(w)^2}+\sqrt{d_G(u)^2+d_G(x^\prime)^2}\nonumber\\
&&+\sqrt{d_G(u)^2+d_G(x)^2}+\sqrt{1+d_G(z)^2}\nonumber\\
&&-\sum_{w\in N_G(u)\setminus \{x,x^\prime\}}\sqrt{(d_G(u)-1)^2+d_G(w)^2}-\sqrt{(d_G(u)-1)^2+d_G(x^\prime)^2}\nonumber\\
&&-\sqrt{2^2+d_G(x)^2}-\sqrt{2^2+d_G(z)^2}\nonumber\\
&&>\sqrt{d_G(u)^2+d_G(x)^2}+\sqrt{1+d_G(z)^2}-\sqrt{2^2+d_G(x)^2}-\sqrt{2^2+d_G(z)^2}.\label{e2}
\end{eqnarray}
 Therefore from the  inequalities (\ref{e1}) or  (\ref{e2}), it follows that
\begin{eqnarray}
SO(G)-SO(G^\prime)>\sqrt{9+d_G(x)^2}-\sqrt{4+d_G(x)^2}-\theta(2)\label{eq1}
\end{eqnarray}
since $d_G(u)\geq 3$,  $d_G(z)\geq 2$ and $\theta(t)$ is decreasing. Clearly $d_G(x)\leq 2$. If $d_G(x)=1$ then we have $SO(G)>SO(G^\prime)$ from (\ref{eq1}). If  $d_G(x)=2$ then we also get
$SO(G)-SO(G^\prime)>\sqrt{13}-\sqrt{8}-\theta(2)>0$  from (\ref{eq1}).
\end{proof}

A tree is said to be star-like if it has exactly one vertex of degree greater than two. Connected graphs of order $n$  with the maximum degree at most two are only $P_n$ and $C_n$. 
In \cite{G},  it has been proved that $SO(G)> SO(P_n)$  for any connected graph  $G$ of order $n$.  Therefore we consider a graph $G$ which is different from  $P_n$ and $C_n$.

\begin{theorem}
Let $G$ be a connected graph of order $n$ with maximum degree $\Delta\geq 3$. Then

\noindent 
\rm{(i)}  If  $2\Delta\leq n-1$ then   
	\begin{equation}
SO(G)\geq \Delta(\sqrt{\Delta^2+4}+\sqrt{5})+2(n-2\Delta-1)\sqrt{2}\label{e01}
\end{equation}
with equality  holds if and only if $G$ is isomorphic to  a star-like tree of order $n$ with maximum degree $\Delta$ in which  all neighbors of the maximum degree vertex have degree two.

\noindent
\rm{(ii)} If  $2\Delta> n-1$ then   
\begin{equation}
SO(G)\geq (n-1-\Delta)(\sqrt{\Delta^2+4}+\sqrt{5})+(2\Delta-n+1)\sqrt{\Delta^2+1}\label{e02}
\end{equation} 
with equality  holds  if and only if $G$ is isomorphic to  a star-like tree of order $n$ with maximum degree $\Delta$ in which the maximum degree vertex has exactly $2\Delta-n+1$ pendent neighbors.
\end{theorem}
\begin{proof}
Let  $SO(G)$ be minimum in the class of graphs of order $n$ with maximum degree $\Delta$ and  $w$ be the maximum degree vertex of $G$.  If there is  a non-cut edge $xy$  in $G$ such that $x\neq w$ and $y\neq w$,  then $SO(G)>SO(G-xy)$ and it follows that $G$ is a tree. 
Now, we prove that $G$ is isomorphic to a star-like tree of order $n$ with maximum degree $\Delta$. If not there is a pendent path  $uu_1\cdots u_k$ such that $u\neq w$.  Clearly there is a pendent vertex $z$ ($\neq u_k$) in $G$. Then  $SO(G)>SO(G-uu_1+u_1z)$ by Lemma \ref{lemma31}  and it contradicts the fact that $SO(G)$ is minimum. 

Hence $G$ is a star-like tree of order $n$ with maximum degree $\Delta$. Let $k$  be the number of pendent neighbors of $w$.  Then 
\begin{equation}
SO(G)=k(\theta(2)-\theta(\Delta))+\Delta(\sqrt{\Delta^2+4}+\sqrt{5})+2(n-1-2\Delta)\sqrt{2}. \label{e10}
\end{equation}
Since $\theta$ is a decreasing function and $\Delta \geq 3$, we have  $\theta(2)>\theta(\Delta)$. Therefore we distinguish the following two cases.

\vspace{3mm}
\noindent
\rm{(i)}  If  $2\Delta\leq n-1$ then there are star-like trees of order $n$ with maximum degree $\Delta$ such that $k=0$. Hence from (\ref{e10}), we obtain the required result.

\vspace{3mm}
\noindent
\rm{(ii)} If  $2\Delta> n-1$ then  $k\geq 2\Delta-n+1$. Hence from (\ref{e10}), we easily get the inequality (\ref{e02})  and with equality if and only if $G$ is isomorphic to  a star-like tree of order $n$ with maximum degree $\Delta$ in which the maximum degree vertex has exactly $2\Delta-n+1$ pendent neighbors.
\end{proof}

Denote by $C_{n, 1}$  the  graph obtained  by attaching one pendent  edge to a vertex of $C_ {n-1}$.
\begin{theorem}\label{th11} Let  $SO(G)$ be minimum in the class of graphs of order $n$ with girth $g$.  If $G$  is different from $C_n$, then $G$ is isomorphic to the unicyclic graph that has  exactly  one pendent path of length at least two.
\end{theorem}
\begin{proof}
Let  $C$ be a cycle of length $g$ and  $xy\notin C$ be a non-cut edge of $G$. Then $SO(G)>SO(G-xy)$ and it follows that $G$ is a unicyclic graph. If $G$ is isomorphic to the unicyclic graph that has  exactly  one pendent path of length at least two, then we have
\begin{equation}
SO(G)\geq \sqrt{5}+3\sqrt{13}+2\sqrt{2}(n-4).\label{eq010}
\end{equation}
 If $G$ is isomorphic to $C_ {n-1}$, then the inequality in (\ref{eq010}) is strict.  Because $SO(C_{n, 1})=2\sqrt{2}(n-3)+2\sqrt{13}+\sqrt{10}$. Otherwise,  repeatedly using  the transformation in Lemma \ref{lemma31}, we get the required result.  
\end{proof}

The following result easily follows from Theorem \ref{th11}. 
\begin{theorem}
Let $G$ be a unicyclic  graph order $n$ which is different from $C_n$.
Then $SO(C_n)<SO(G)$.
\end{theorem}	
\begin{proof}
Let $g$ be the girth of $G$. Since $G$ is different from $C_n$, we have 
 $SO(G)\geq \sqrt{5}+3\sqrt{13}+2\sqrt{2}(n-4)$ by Theorem \ref{th11}. From this, we get the required result because $SO(C_n)=2n\sqrt{2}$.
\end{proof}

\section{Graphs with maximum Sombor index}

In this section, we study the graphs with maximum Sombor index. Namely, we obtain the sharp upper  bounds  on $SO$ index of a unicyclic  graph of order $n$ with girth $g$ and  of a  graph of order $n$ with $k$ pendent vertices ($r$ cut edges). 

\begin{lemma}\label{lemma22} ~Let $G$ be a connected graph and  $uv$ be a non-pendent cut edge in $G$.  Denote by $G^\prime$  the graph obtained by the contraction of $uv$ onto the vertex $u$ and adding a pendent vertex $v$ to $u$. Then  $SO(G)< SO(G^\prime)$.
\end{lemma}
\begin{proof} Let $N_{G}(u)\setminus \{v\}=\{u_1,u_2,\dots, u_s\}$ and $N_{G}(v)\setminus \{u\}=\{v_1,v_2,\dots, v_t\}$, then $d_G(u)=s+1$ and $d_G(v)=t+1$. 
Since $uv$ is a non-pendent cut edge of $G$, we have $st>0$.  Hence, by the definition of $SO$, we obtain
	\begin{eqnarray}SO(G^\prime)-SO(G)&=&\sum_{i=1}^s\sqrt{(s+t+1)^2+d_{G}(u_i)^2}-\sum_{i=1}^s\sqrt{(s+1)^2+d_{G}(u_i)^2}\nonumber\\
	&+&\sum_{j=1}^t\sqrt{(s+t+1)^2+d_{G}(v_j)^2}-\sum_{j=1}^t\sqrt{(t+1)^2+d_{G}(v_j)^2}\nonumber\\
	&+&\sqrt{(s+t+1)^2+1}-\sqrt{(s+1)^2+(t+1)^2}\nonumber\\
	&>&\sqrt{(s+t+1)^2+1}-\sqrt{(s+1)^2+(t+1)^2}\nonumber
	\end{eqnarray}
	and it follows that $SO(G)< SO(G^\prime)$ since $[(s+t+1)^2+1]-[(s+1)^2+(t+1)^2]=2st>0$.
\end{proof}
\begin{proposition}\label{prop1}
	Let $G$ be a connected graph of order $n$ with $k$ cut edges.  If $SO(G)$  is maximum in the class of graphs of order $n$ with $k$ cut edges, then all $k$ cut edges of $G$ are pendent. 
\end{proposition}
\begin{proof}
Suppose, on the contrary, that $G$ contains a non-pendent cut edge $uv$. Let $G^\prime$  be the graph obtained by the contraction of $uv$ onto the vertex $u$ and adding a pendent vertex $v$ to $u$. Then  $SO(G)< SO(G^\prime)$ by Lemma \ref{lemma22}. Therefore, we have a contradiction to the assumption that $SO(G)$  is maximum  in the class of graphs of order $n$ with $k$ cut edges.
\end{proof}

Let $A=(a_1, a_2, \ldots, a_n)$ and $B=(b_1, b_2, \ldots, b_n)$  be  non-increasing  two sequences on an interval $I$ of real numbers such that $a_1+a_2+\cdots+a_n=b_1+b_2+\cdots+b_n$. If 
$$a_1+a_2+\cdots+a_i \geq b_1+b_2+\cdots+b_i ~~\text{for all}~~  1\leq i\leq n-1$$
 then we say that  $A$ majorizes  $B$.

\begin{lemma} {\rm (Karamata's inequality)} \label{lem-8}
	Let $f\colon I\rightarrow \mathbb{R}$ be a strictly convex function.  Let $A=(a_1, a_2, \ldots, a_n)$ and $B=(b_1, b_2, \ldots, b_n)$  be  non-increasing  sequences on $I$.  If $A$ majorizes  $B$ then
	$$f(a_{1})+f(a_2)+\cdots +f(x_{n})\geq f(b_{1})+f(b_2)+\cdots +f(b_{n})$$
with equality   if and only if  $a_i=b_i$ for all $ 1\leq i \leq n$.
\end{lemma}

\begin{theorem}\label{theorem42} Let  $G$ be a unicyclic graph of order $n$ with girth $g$. Then
\begin{equation}
SO(G)\leq 2\sqrt{(n-g+2)^2+4}+(n-g)\sqrt{(n-g+2)^2+1}+2\sqrt{2}(g-2)\label{eq8}
\end{equation}
with equality holds if and only if $G$ is isomorphic to the graph obtained by attaching $n-g$ pendent edges to a vertex of $C_{g}$.
\end{theorem}
 \begin{proof}
Denote by $U_{n,g}$  the  graph obtained by attaching $n-g$ pendent edges to a vertex of $C_{g}$. If $G$ is isomorphic to  $U_{n,g}$ then the equality holds in (\ref{eq8}).
Suppose that $G$ is not isomorphic to  this graph and  $SO(G)$ is maximum among all unicyclic graphs of order $n$ with girth $g$. Then by Proposition \ref{prop1},  $G$ is isomorphic to a graph such that each pendent edge is attached to the unique cycle. Denote (in clockwise order) by $u_1, u_2, \ldots, u_g$ the vertices on the cycle.   Let $k$ be the number of pendent edges in  $G$.  For simplicity's sake we denote $d_G(u_i)=d_i$, $i=1, 2, \dots, g$.  Then, we have
\begin{equation}
2\leq d_i\leq k+2~~\text{and}~~d_1+d_2+\cdots+d_g=k+2g=n+g.\label{eq11}
\end{equation}

Consider  a non-increasing sequence $A=\{a_i\}$ with length $n+g$ as follows:
\begin{eqnarray}
\underbrace{\frac{k+2}{2},\dfrac{k+2}{2}}_{2},~\underbrace{1, 1, \dots, 1,}_{2g-4}\,~\underbrace{\frac{2}{k+2},\frac{2}{k+2},\dots,\frac{2}{k+2}}_{k+2}.\nonumber
\end{eqnarray}
Let $c_1, c_2, \ldots, c_g$ be a permutation of the sequence $d_1, d_2, \ldots, d_g$.
Then,  we consider  a non-increasing sequence $B=\{b_i\}$ with length $n+g$ as follows:
\begin{eqnarray}~\underbrace{\frac{c_2}{c_1},\dots,\frac{c_2}{c_1}}_{c_1},~\underbrace{\frac{c_3}{c_2},\dots,\frac{c_3}{c_2}}_{c_2},~\dots,
~\underbrace{\frac{c_{g}}{c_{g-1}},\dots,\frac{c_{g}}{c_{g-1}}}_{c_{g-1}},~\underbrace{\frac{c_1}{c_g},\dots,\frac{c_1}{c_g}}_{c_g},\nonumber
\end{eqnarray}
where for all $1\leq i\leq g$ there exists $j$ such that  $c_i/c_{i-1}=d_j/d_{j-1}$ with $c_0=c_g$ and $d_0=d_g$.

\vspace{3mm}
Now we prove that $A$ majorizes $B$.  Denote  $A_i=a_1+a_2+\cdots+a_i$ and $B_i=b_1+b_2+\cdots+b_i$ for $1\leq i \leq n+g$.   Then, one can easily see that $A_{n+g}=B_{n+g}=n+g$, $A_1\geq B_1$ and $A_2\geq B_2$ from (\ref{eq11}) because  $c_1, c_2, \ldots, c_g$ is a permutation of $d_1, d_2, \ldots, d_g$.

\vspace{3mm}
Suppose first that $3\leq i\leq 2g-2$. Then, we have $A_i=k+2+i-2=k+i$ and
$$B_i=c_2+c_3+\cdots+c_{s-1}+\frac{pc_{s}}{c_{s-1}}~~ \text{ for some positive integers}~s\text{ and}~p,$$
such that $c_1+\cdots+c_{s-2}+p=i$ and $p\leq c_{s-1}$. Therefore, we get
\begin{equation}
A_i-B_i=k+c_1+p-c_{s-1}-\frac{pc_{s}}{c_{s-1}}.\label{eq121}
\end{equation}
On the other hand,  for $1\leq i<j\leq g$ we have $d_i+d_j\leq k+4$ and it follows that $c_i+c_j\leq k+4$.
If $p\geq 2$, then $c_{s-1}+pc_{s}/c_{s-1}\leq c_{s-1}+c_{s}\leq k+4\leq k+c_1+p$ since $p\leq c_{s-1}$ and $c_1\geq 2$.  Therefore, we have $A_i\geq B_i$ from (\ref{eq121}). If $p=1$, from  (\ref{eq121}), we get
\begin{equation}
A_i-B_i\geq k+3-c_{s-1}-\frac{c_{s}}{2}=k+3+\frac{c_{s}}{2}-(c_{s-1}+c_{s})\geq 0\label{eq12}
\end{equation}
since $c_1, c_{s-1}, c_s\geq 2$ and $c_{s-1}+c_{s}\leq k+4$.

\vspace{3mm}
Suppose now that $2g-2< i\leq n+g$. Then since $A_{n+g}=n+g$,
\begin{equation}
A_i=n+g-(n+g-i)\cdot \frac{2}{k+2}.\label{eq13}
\end{equation}

Moreover, since $B_{n+g}=n+g$ and the sequence $B$ is non-increasing, we get
\begin{eqnarray}
B_{i}\leq n+g-(n+g-i)\frac{c_1}{c_g}.\label{eq14}
\end{eqnarray}
Therefore, from (\ref{eq13}) and  (\ref{eq14}) we get $A_i\geq B_i$  using $2\leq c_1, c_g\leq k+2$. Hence we conclude that $A$ majorizes $B$.

\vspace{3mm}
Now, we prove that $SO(G)<SO(U_{n,g})$ by using well-known Karamata's inequality. For this purpose, let us consider a function $f(x)=\sqrt{1+x^2}$  and  it is easy to see that this function is strictly convex for $x\in [0,+\infty)$. By the definition of $SO(G)$ and  $G$ is not isomorphic to  $U_{n, g}$, we obtain
\begin{eqnarray}
&&SO(G)=\nonumber\\
&&=(d_1-2)\sqrt{d_1^2+1}+\cdots+(d_g-2)\sqrt{d_g^2+1}
+\sqrt{d_1^2+d_2^2}+\cdots+\sqrt{d_{g-1}^2+d_g^2}+\sqrt{d_g^2+d_1^2}\nonumber\\
&&< k\sqrt{(k+2)^2+1}+d_1\sqrt{1+\left(\frac{d_2}{d_1}\right)^2}+\cdots+d_{g-1}\sqrt{1+\left(\frac{d_g}{d_{g-1}}\right)^2}+d_g\sqrt{1+\left(\frac{d_1}{d_g}\right)^2}\nonumber\\
&&= k\sqrt{(k+2)^2+1}+d_1f\left(\frac{d_2}{d_1}\right)+\cdots+d_{g-1}f\left(\frac{d_g}{d_{g-1}}\right)+d_g f\left(\frac{d_1}{d_g}\right)\label{eq15}
\end{eqnarray}
by  (\ref{eq11}). Without loss of generality we may assume that
$$\frac{d_2}{d_1}\geq \frac{d_3}{d_2} \geq \dots \geq \frac{d_g}{d_{g-1}} \geq  \frac{d_1}{d_g}.$$
Then we have  proved that $A$ majorizes the sequence
\begin{eqnarray}~\underbrace{\frac{d_2}{d_1},\dots,\frac{d_2}{d_1}}_{d_1},~\underbrace{\frac{d_3}{d_2},\dots,\frac{d_3}{d_2}}_{d_2},~\dots,
~\underbrace{\frac{d_{g}}{d_{g-1}},\dots,\frac{d_{g}}{d_{g-1}}}_{d_{g-1}},~\underbrace{\frac{d_1}{d_g},\dots,\frac{d_1}{d_g}}_{d_g}. \nonumber
\end{eqnarray}
Therefore from (\ref{eq15}), we get the required strict inequality in (\ref{eq8})
by Karamata's inequality.
\end{proof}

\begin{lemma} \label{lemma54} If $x\geq y\geq 0$ and $a\geq 1$  then
	\begin{eqnarray}(x+1)\sqrt{(x+a)^2+1}+y\sqrt{(y+a-1)^2+1}\geq x\sqrt{(x+a-1)^2+1}+(y+1)\sqrt{(y+a)^2+1}.\nonumber
	\end{eqnarray}
\end{lemma}
\begin{proof} Let us consider a function
	\begin{eqnarray}
	\phi(x)=(x+1)\sqrt{(x+a)^2+1}-x\sqrt{(x+a-1)^2+1},~~~x\in[0,+\infty).\nonumber
	\end{eqnarray}
	Then, we have
	\begin{eqnarray}\phi^\prime(x)&=&\sqrt{(x+a)^2+1}+\frac{(x+1)(x+a)}{\sqrt{(x+a)^2+1}}-\sqrt{(x+a-1)^2+1}-\frac{x(x+a-1)}{\sqrt{(x+a-1)^2+1}}\nonumber\\
	&>&\frac{(x+1)(x+a)}{\sqrt{(x+a)^2+1}}-\frac{x(x+a-1)}{\sqrt{(x+a-1)^2+1}}\nonumber\\
	&>&\frac{a+2x}{\sqrt{(x+a-1)^2+1}}>0\nonumber
	\label{eq531}
	\end{eqnarray}
and it follows that $\phi(x)$ is an increasing function. Therefore we get the required  
inequality since $x\geq y$.	
\end{proof}

\begin{theorem}\label{theorem57} Let  $G$ be a connected graph of order $n$ with $k$ pendent vertices. Then 
	\begin{eqnarray}SO(G)\leq \frac{(n-k-2)(n-k-1)^2}{\sqrt{2}}+k\sqrt{(n-1)^2+1}+(n-k-1)\sqrt{(n-1)^2+(n-k-1)^2}\nonumber
	\end{eqnarray}
	with equality  if and only if $G$ is isomorphic to a graph obtained by attaching $k$ pendent edges to a vertex of $K_{n-k}$.
\end{theorem}
\begin{proof} If $G$ is isomorphic to  a graph obtained by attaching $k$ pendent edges to a vertex of $K_{n-k}$, then  equality holds in the inequality of the statement of the theorem.  Suppose that $G$ is not isomorphic to  this graph and  $SO(G)$ is maximum among all  graphs of order $n$ with $k$ pendent vertices. Then by Proposition \ref{prop1},  $G$ is isomorphic to a graph such that each pendent edge is attached to the clique with $n-k$ vertices. Denote by $u_1, u_2, \ldots, u_{n-k}$ the vertices of the clique.   Denote $d_G(u_i)=d_i$, $i=1, 2, \dots, n-k$.  Without loss of generality we may assume that	
$d_1\geq d_2\geq \cdots \geq d_{n-k}.$ Then, we have 
\begin{equation}
d_1+d_2+\cdots+d_{n-k}=k+(n-k)(n-k-1)~~\text{and}~~n-k-1\leq d_i< n-1\label{eq001}
\end{equation}
Assume that $d_t=\min\{d_i\mid n-k-1< d_i<n-1\}$. Then there is a pendent edge $u_tx$ in $G$ and consider the graph $G^\prime= G-u_tx+u_1x$. 
If we set $x=d_1-n+k+1$, $a=n-k$ and $y=d_t-n+k$ in the inequality of the statement of Lemma \ref{lemma54}, then 
\begin{eqnarray}
(d_1-n+k+2)\sqrt{(d_1+1)^2+1}+(d_t-n+k)\sqrt{(d_t-1)^2+1}\nonumber\\
\geq (d_1-n+k+1)\sqrt{d_1^2+1}+(d_t-n+k+1)\sqrt{d_t^2+1}.\label{eqn1}
\end{eqnarray}
Therefore, we have	
	\begin{eqnarray}SO(G^\prime)-SO(G)&=&\sum_{i\neq 1, t}\sqrt{(d_1+1)^2+d_i^2}+\sum_{i\neq 1, t}\sqrt{(d_t-1)^2+d_i^2}+\sqrt{(d_1+1)^2+(d_t-1)^2}\nonumber\\
	&&+(d_1-n+k+2)\sqrt{(d_1+1)^2+1}+(d_t-n+k)\sqrt{(d_t-1)^2+1}\nonumber\\
    &&-\sum_{i\neq 1, t}\sqrt{d_1^2+d_i^2}-\sum_{i\neq 1, t}\sqrt{d_t^2+d_i^2}-\sqrt{d_1^2+d_t^2}\nonumber\\
    &&-(d_1-n+k+1)\sqrt{d_1^2+1}-(d_t-n+k+1)\sqrt{d_t^2+1}\nonumber\\
	&>&\sum_{i\neq 1, t}d_i\sqrt{1+\left(\frac{d_1+1}{d_i}\right)^2}+\sum_{i\neq 1, t}d_i\sqrt{1+\left(\frac{d_t-1}{d_i}\right)^2}\nonumber\\
	&&-\sum_{i\neq 1, t}d_i\sqrt{1+\left(\frac{d_1}{d_i}\right)^2}-\sum_{i\neq 1, t}d_i\sqrt{1+\left(\frac{d_t}{d_i}\right)^2} \label{eq511}
	\end{eqnarray}
by (\ref{eqn1}) and $\sqrt{(d_1+1)^2+(d_t-1)^2}\geq \sqrt{d_1^2+d_t^2}$.

\vspace{3mm}
	Consider non-increasing two sequences $A=\{a_i\}$ and $B=\{b_i\}$ as follows:
	\begin{eqnarray}
	A:&~ \underbrace{\frac{d_1+1}{d_{n-k}},\cdots, \frac{d_1+1}{d_{n-k}}}_{d_{n-k}}, \cdots, \underbrace{\frac{d_1+1}{d_{t+1}}, \cdots, \frac{d_1+1}{d_{t+1}}}_{d_{t+1}}, \underbrace{\frac{d_1+1}{d_{t-1}}, \cdots, \frac{d_1+1}{d_{t-1}}}_{d_{t-1}}, \cdots, \underbrace{\frac{d_1+1}{d_{2}}, \cdots, \frac{d_1+1}{d_{2}}}_{d_{2}},\nonumber\\
	&~ \underbrace{\frac{d_t-1}{d_{n-k}},\cdots, \frac{d_t-1}{d_{n-k}}}_{d_{n-k}}, \cdots, \underbrace{\frac{d_t-1}{d_{t+1}}, \cdots, \frac{d_t-1}{d_{t+1}}}_{d_{t+1}}, \underbrace{\frac{d_t-1}{d_{t-1}}, \cdots, \frac{d_t-1}{d_{t-1}}}_{d_{t-1}}, \cdots, \underbrace{\frac{d_t-1}{d_{2}}, \cdots, \frac{d_t-1}{d_{2}}}_{d_{2}},\nonumber\\
	B:&\underbrace{\frac{d_1}{d_{n-k}}, \cdots, \frac{d_1}{d_{n-k}}}_{d_{n-k}}, \cdots,  \underbrace{\frac{d_1}{d_{t+1}}, \cdots, \frac{d_1}{d_{t+1}}}_{d_{t+1}}, \underbrace{\frac{d_1}{d_{t-1}}, \cdots, \frac{d_1}{d_{t-1}}}_{d_{t-1}}, \cdots,  \underbrace{\frac{d_1}{d_{2}}, \cdots, \frac{d_1}{d_{2}}}_{d_{2}}, \nonumber\\
	&\underbrace{\frac{d_t}{d_{n-k}}, \cdots, \frac{d_t}{d_{n-k}}}_{d_{n-k}}, \cdots,  \underbrace{\frac{d_t}{d_{t+1}}, \cdots, \frac{d_t}{d_{t+1}}}_{d_{t+1}}, \underbrace{\frac{d_t}{d_{t-1}}, \cdots, \frac{d_t}{d_{t-1}}}_{d_{t-1}}, \cdots,  \underbrace{\frac{d_t}{d_{2}}, \cdots, \frac{d_t}{d_{2}}}_{d_{2}}.\nonumber
	\end{eqnarray}
	
Denote  $A_i=a_1+a_2+\cdots+a_i$ and $B_i=b_1+b_2+\cdots+b_i$ for $1\leq i\leq 2\sum_{i\neq 1, t}d_i$.  From the above, it is easy to see that both the summations of all elements of $A$ and $B$ are equal to $(n-k-2)(d_1+d_t)$, and $A_i\geq B_i$  for all $1\leq i\leq 2\sum_{i\neq 1, t}d_i$.  Hence $A$ majorizes $B$. 

On the other hand,   $f(x)=\sqrt{1+x^2}$ is  a strictly convex function  on  $[0,+\infty)$. Therefore,  using Karamata's inequality in  (\ref{eq511}), we get
$SO(G^\prime)>SO(G)$ and it contradicts  the fact that $SO(G)$ is maximum among all  graphs of order $n$ with $k$ pendent vertices. 
\end{proof} 

The same argument as in the proof of Theorem \ref{theorem57} yields the following result.
\begin{theorem}  If  $SO(G)$ is  maximum  in the class of connected graphs  of order $n$ with $r$ cut edges, then $G$ is isomorphic to the graph  obtained by attaching $r$ pendent edges to a vertex of $K_{n-r}$.
\end{theorem}

\noindent
\textbf{Acknowledgment:}
The authors would like to express our very great appreciation to Prof. Ivan Gutman for introducing this topic and sending the paper \cite{G}. The first author is grateful  for the financial support of MNUE. The second author was supported by Project of Inner Mongolia University for Nationalities Research Funded Project (NMDYB17155).

\end{document}